\documentclass[reqno, 12pt]{amsart}
   
\usepackage{amsmath,amssymb,amsthm}
\usepackage{dsfont}

\usepackage{xcolor}
\usepackage[backref]{hyperref}
\hypersetup{
    colorlinks,
    linkcolor={red!60!black},
    citecolor={green!60!black},
    urlcolor={blue!60!black}
}

\usepackage{graphicx}
    \usepackage{tikz}
    \usetikzlibrary{arrows}
\usepackage{verbatim}

\usepackage[open,openlevel=2,atend]{bookmark}

\usepackage[abbrev,msc-links,backrefs]{amsrefs} 
\usepackage{doi}

\renewcommand{\PrintDOI}[1]{\doi{#1}}

\usepackage[T1]{fontenc}
\usepackage{lmodern}
\usepackage[babel]{microtype}
\usepackage[english]{babel}

\usepackage{enumitem}

\usepackage{geometry}
\numberwithin{equation}{section}
\numberwithin{figure}{section}

\usepackage{thmtools, thm-restate}

\theoremstyle{plain}
\newtheorem{thm}{Theorem}[section]

\newtheorem{lem}[thm]{Lemma}
\newtheorem{corollary}[thm]{Corollary}
\newtheorem{cor}[thm]{Corollary}

\newtheorem*{claim*}{Claim}
\newtheorem{claim}{Claim}[]
\newtheorem{thm-intro}{Theorem}[]
\newtheorem{conj-intro}[thm-intro]{Conjecture}
\newtheorem*{meta-question*}{Meta question}
\newtheorem*{thm*}{Theorem}

\theoremstyle{definition}
\newtheorem{definition}[thm]{Definition}

\newtheorem{ques}[thm]{Question}

\newtheorem*{example*}{Example}

\DeclareMathOperator{\ingoing}{in}
\DeclareMathOperator{\outgoing}{out}

\newcommand{\orthogonal}{\perp}

\title[]{Enlarging vertex-flames in countable digraphs}
\author[J.~Erde \and J.P.~Gollin \and A.~Jo\'{o}]{Joshua Erde \and J.~Pascal Gollin \and Attila Jo\'{o}}

\address{Joshua Erde, Graz University of Technology, Institute of Discrete Mathematics, Steyrergasse 30, 8010 Graz, Austria}
\email{erde@math.tugraz.at}

\address{J. Pascal Gollin, Discrete Mathematics Group, Institute for Basic Science (IBS), 55 Expo-ro, Yuseong-gu, Daejeon, Korea, 34126}
\email{\tt pascalgollin@ibs.re.kr}
\thanks{The second author was supported by the Institute for Basic Science (IBS-R029-C1).}

\address{Attila Jo\'{o},
University of Hamburg, Department of Mathematics, Bundesstra{\ss}e 55 (Geomatikum), 20146 Hamburg, Germany and 
Alfr\'{e}d R\'{e}nyi Institute of Mathematics, Set theory and general topology research division, 13-15 Re\'{a}ltanoda St., 
Budapest, Hungary}
\email{attila.joo@uni-hamburg.de,  jooattila@renyi.hu}
\thanks{The third author would like to thank the generous support of the Alexander von Humboldt Foundation and NKFIH OTKA-129211.}

\begin{document}

\maketitle

\begin{abstract}
    A rooted digraph is a vertex-flame if for every vertex~$v$ there is a set of internally disjoint directed paths from the root to~$v$ whose set of terminal edges covers all ingoing edges of~$v$. 
    It was shown by Lov\'{a}sz that every finite rooted digraph admits a spanning subdigraph which is a vertex-flame and large, 
    where the latter means that it preserves the local connectivity to each vertex from the root. Calvillo-Vives rediscovered and 
    extended this theorem proving that every vertex-flame of a given finite rooted digraph can be extended to be large. The analogue of Lov\'{a}sz' result for countable digraphs was shown by the third author where the notion of largeness is interpreted in a structural way as in the infinite version of Menger's theorem.
    We give a common generalisation of this and Calvillo-Vives' result by showing that in every countable rooted digraph each vertex-flame can be extended to a large vertex-flame.
    
\end{abstract}

\section{Introduction}

Given a rooted digraph~$D$ with root~$r$ the \emph{local connectivity} from~$r$ to~$v$, written~${\kappa_D(r,v)}$, is the size of the largest set of internally disjoint paths from~$r$ to~$v$ in~$D$. 
Lov\'{a}sz showed in~\cite{lovasz} that there always exists a spanning subgraph which preserves all the local connectivities from~$r$, whilst keeping only~${\kappa_D(r,v)}$ ingoing edges at each~$v$.

More precisely, we say that a rooted digraph is a \emph{vertex-flame} (or shortly \emph{flame}) 
if for every vertex~$v$ (other than the root) there is a set of internally disjoint directed paths from the root to~$v$ whose set of terminal edges is~${\ingoing_D(v)}$ (i.e., the set of ingoing edges of~$v$). 
Lov\'{a}sz showed that every finite rooted digraph~$D$ contains a vertex-flame which is \emph{large}, i.e., it has the same local connectivity from the root to any other vertex as in~$D$.\footnote{Edge-flames can be defined in an analogous manner by replacing `internally disjoint' with `edge-disjoint'. 
In this case, preservation of local edge-connectivities from the root in an edge-flame can be accomplished. 
In fact Lov\'{a}sz  originally proved the edge variant, from which the vertex version can be deduced as a corollary.}

\begin{thm}[Lov\'{a}sz] \label{t:finiteflame}
    If~$D$ is a finite digraph with~${r \in V(D)}$, then there is a spanning subgraph~$F$ of~$D$ such that for every~${v \in V(D)-r}$, 
    \[
        \kappa_D(r,v) = \kappa_{F}(r,v) = |\ingoing_{F}(v)|.
    \]
\end{thm}

Theorem \ref{t:finiteflame} above was rediscovered by Calvillo-Vives in~\cite{calvillo-vives} in a stronger form:

\begin{thm}[Calvillo-Vives]
    \label{thm:finextend}
    Any vertex-flame in a given finite rooted digraph can be extended to a large one. 
\end{thm}
Since in particular the spanning subgraph with empty edge set is a flame, this clearly implies Theorem~\ref{t:finiteflame}.

Recently, the third author proved a generalisation of Theorem~\ref{t:finiteflame} for countably infinite digraphs~\cite{attila-flames}. 
As it is common in the context of infinite graphs, a naive cardinality-based approach yields a much weaker result than a more structural generalisation of the problem. 
For example, Erd\H{o}s conjectured a structural generalisation of Menger's theorem for infinite graphs, which was eventually proved in an influential paper of Aharoni and Berger~\cite{infinite-menger}. 
In particular they showed that in every digraph~$D$ for every~${x \neq y \in V(D)}$ with~${xy \notin E(D)}$ 
there is an \emph{orthogonal pair} of a set~$\mathcal{P}$ of internally disjoint directed paths from~$x$ to~$y$ 
and an ${S \subseteq V(D)\setminus \{ x,y\} }$ separating~$y$ from~$x$ (i.e., meeting every path from~$x$ to~$y$), 
where orthogonality means that~$S$ contains exactly one vertex from each path in~$\mathcal{P}$, and no other vertices.

Motivated by this, the third author gave a structural characterisation of when it can be said that a subgraph preserves the local connectivity from the root to each vertex in an infinite digraph in the spirit of the Aharoni-Berger theorem. 
More precisely we say that a spanning subdigraph~$L$ of a rooted digraph~$D$ is $D$-\emph{large} (or \emph{large}) if for every~${v \neq r}$ the digraphs~${D-rv}$ and~${L-rv}$ share such an orthogonal pair for~$r$ and~$v$ and furthermore~${rv \in E(L)}$ if~${rv \in E(D)}$. 
For more details on the definition, see Section~\ref{sec:prelims}. 

With this generalisation, the third author showed the following.

\begin{thm}
   \cite{attila-flames}*{Theorem~1.2}\label{t:atillaflame}
   Every countable rooted digraph contains a large vertex-flame.
\end{thm}
After this infinite generalisation of Theorem~\ref{t:finiteflame} it was an open question if Theorem~\ref{thm:finextend} remains true in infinite settings. 
Our main result is to show that this is indeed the case.

\begin{restatable}{thm}{extend}\label{thm:extend}
    Let~$F$ be a vertex-flame in a countable rooted digraph~${D}$. 
    Then there is a large vertex-flame~$F^*$ in $D$ with~${F \subseteq F^*}$.
\end{restatable}

While Theorem~\ref{t:atillaflame} only guarantees the existence of some large flame, our main result ensures that every flame can be extended to a large flame, and hence shows that every subgraph-maximal flame is large.
It is worth mentioning that a subgraph-maximal flame may not always exist. 
Indeed, consider the digraph in Figure~\ref{f:ex-flame} rooted in~$r$ with disjoint sets~$V_0$,~$V_1$,$V_2$ of vertices, where~$V_0$ and~$V_2$ are countably infinite and~$V_1$ is uncountable, which contains all edges of the form~$rv$ for~${v \in V_0}$ as well as for~${i \in \{0,1\}}$ all edges of the from~$vw$ for all~${v \in V_i}$ and~${w \in V_{i+1}}$. 
A large subgraph must contain all the outgoing edges of~$r$. 
For~${v \in V_1}$, the set of all~$(r,v)$-paths is the only Erd\H{o}s-Menger path-system, 
and thus a large subgraph needs to keep all the edges between~$V_0$ and~$V_1$. 
Finally, for~${v \in V_2}$ the only Erd\H{o}s-Menger separation is~$V_0$. 
To be large, infinitely many ingoing edges for every~${v \in V_2}$ must be kept. 
The resulting large subgraph is a flame if and only if for every~${v \in V_2}$ the set of the kept ingoing edges is only countably infinite.  
Therefore there is no maximal large flame, and hence no maximal flame.

    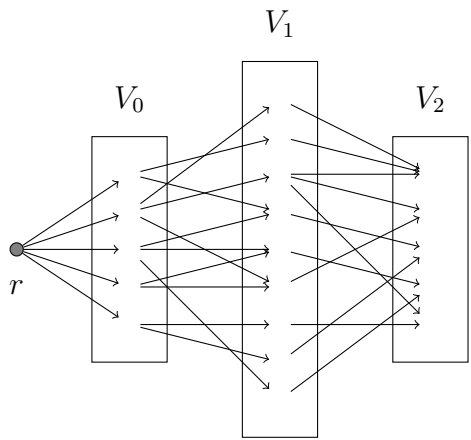
\begin{figure}[htbp]
        \centering
        \begin{tikzpicture}
            \node[circle, draw, fill=black!50, inner sep=0pt, minimum width=5pt] (v1) at (-0.5,-0.5) {};
            \node at (-0.5,-1) {$r$};
            \draw (0.5,1) rectangle (1.5,-2);
            \draw (2.5,2) rectangle (3.5,-3);
            \draw (4.5,1) rectangle (5.5,-2);
            \node at (1,1.5) {$V_0$};
            \node at (3,2.5) {$V_1$};
            \node at (5,1.5) {$V_2$};
            
            \node (v2) at (1,0.5) {};
            \node (v3) at (1,0) {};
            \node (v4) at (1,-0.5) {};
            \node (v5) at (1,-1) {};
            \node (v6) at (1,-1.5) {};
            \node (v15) at (3,1.5) {};
            \node (v7) at (3,1) {};
            \node (v11) at (3,0.5) {};
            \node (v8) at (3,0) {};
            \node (v12) at (3,-0.5) {};
            \node (v9) at (3,-1) {};
            \node (v13) at (3,-1.5) {};
            \node (v10) at (3,-2) {};
            \node (v14) at (3,-2.5) {};
            \node (v16) at (5,0.5) {};
            \node (v19) at (5,0) {};
            \node (v18) at (5,-0.5) {};
            \node (v20) at (5,-1) {};
            \node (v17) at (5,-1.5) {};
            
            \draw[->] (v1) edge (v2);
            \draw[->] (v1) edge (v3);
            \draw[->] (v1) edge (v4);
            \draw[->] (v1) edge (v5);
            \draw[->] (v1) edge (v6);
            
            \draw[->] (v2) edge (v7);
            \draw[->] (v5) edge (v9);
            \draw[->] (v3) edge (v11);
            \draw[->] (v4) edge (v12);
            \draw[->] (v6) edge (v13);
            \draw[->] (v2) edge (v8);
            \draw[->] (v5) edge (v12);
            \draw[->] (v3) edge (v9);
            \draw[->] (v6) edge (v10);
            \draw[->] (v4) edge (v14);
            \draw[->] (v3) edge (v15);
            \draw[->] (v4) edge (v8);
            \draw[->] (v15) edge (v16);
            \draw[->] (v7) edge (v16);
            \draw[->] (v13) edge (v17);
            \draw[->] (v8) edge (v18);
            \draw[->] (v11) edge (v19);
            \draw[->] (v12) edge (v20);
            \draw[->] (v11) edge (v17);
            \draw[->] (v11) edge (v16);
            \draw[->] (v10) edge (v18);
            \draw[->] (v14) edge (v20);
            \draw[->] (v9) edge (v19);
        \end{tikzpicture}
        \label{f:ex-flame}
        \caption{A digraph without a subgraph-maximal flame.}
    \end{figure} 
    
For the proof of Theorem~\ref{thm:extend} we need some lemmas from~\cite{attila-flames} as well as a new framework around the notion of \emph{incompressibility}, which play a key role for the proof. 
We say a set of vertices~$X$ is \emph{incompressible} to another set of vertices~$Y$ if there is a system of disjoint directed paths from~$X$ to~$Y$ that covers all the vertices of~$X$ and for every such path-system it necessarily covers all the vertices of~$Y$ as well. 
Clearly when~$X$ and~$Y$ are finite, the second part of the definition is equivalent to the condition that~${|X| = |Y|}$, however for infinite sets of vertices this is a more delicate structural condition. 

We note that it is still open as to whether the condition that the digraph is countable can be removed from Theorem~\ref{t:atillaflame}.
\begin{ques}
    Does every rooted digraph contain a large vertex-flame? 
\end{ques}
Preservation of edge-connectivity from the root can be also generalised structurally. 
An~${L \subseteq D}$ is edge-large if it contains for each~${v \neq r}$ a system~$\mathcal{P}$ of edge-disjoint paths such that there is a transversal~$C$ for ${\{ E(P) \colon P \in \mathcal{P} \}}$ which is an $rv$-cut in~$D$. 
Although the proofs of the vertex and edge variants are analogous in the finite case, the edge version seems to be strictly harder in infinite digraphs. 
Even the countable case is wide open. 
\begin{ques}
    Does every countable rooted digraph contain an edge-large edge-flame?
\end{ques}

The structure of this paper is as follows. 
In Section~\ref{sec:prelims} we give some basic definitions and citations for some theorems that we use. 
In particular in Section~\ref{subsec:quasi-flames} we will introduce the notion of a $G$-quasi-flame and state a key lemma, using which, in Section~\ref{sec:proof}, we will prove Theorem~\ref{thm:extend}. 
In Section~\ref{sec:incompressibility} we will introduce the concept of incompressibility and develop the tools that we need to prove the key lemma in Section~\ref{sec:key-lemma}.

\section{Preliminaries}
\label{sec:prelims}

\subsection{Basic notation}
\label{subsec:notation} \ 

In this paper, $D$ will denote a (usually infinite) digraph with vertex set~${V(D)}$ and edge set~${E(D)}$. 
For the questions we are considering we may assume without loss of generality that~$D$ does not contain any parallel edges or loops. 

We denote an edge~$e$ of~$D$ directed from~$v$ to~$w$ also by the string~${vw}$. 
Here,~$v$ denotes the \emph{tail} of~$e$ and~$w$ denotes the \emph{head} of~$e$. 
Often we will consider~$D$ as a \emph{rooted digraph}, i.e., a digraph with one of its vertices~$r$ marked as a root. 
For simplicity we may always assume that the root has no ingoing edges, since such edges are not contained in any directed path starting from the root, and hence they are neither contained in any flame nor relevant for the notion of largeness.
Since the root will be the same vertex~$r$ in the whole paper, we omit it from our notation without risking any confusion.

Given a vertex~${v \in V(D)}$, we denote by~$\ingoing_D(v)$ and~$\outgoing_D(v)$ the set of ingoing and outgoing edges of~$v$ respectively. 
Here, and for other notation given in this section, we introduce the subscript since we will often need to work inside subdigraphs of~$D$.

For a vertex or edge~$x$ we denote by~${D - x}$ the digraph obtained from~$D$ by deleting~$x$. 
Similarly for a set~$X$ of either vertices or edges we denote by~${D - X}$ the digraph obtained from~$D$ by deleting~$X$. 
Lastly, given a set~$X$ and an element~${x \in X}$, we denote by~${X - x}$ the set~${X \setminus \{x\}}$.

Let~${v \in V(D)}$ and ${I \subseteq \ingoing_D(v)}$. 
We define the \emph{restriction of~$D$ to~$I$ at~$v$} as the digraph~${D - (\ingoing_D(v) \setminus I)}$ and denote it by~${D \upharpoonright_v I}$.

\subsection{Path-systems}
\label{subsec:path-systems} \ 

Let~${x,y \in V(D)}$ be distinct and let~${X, Y \subseteq V(D)}$. 
We introduce the following notation for paths.

\begin{itemize}
    \item An \emph{${(x, y)}$-path} is a directed path with initial vertex~$x$ and terminal vertex~$y$. 
        We call the ${(x,y)}$-path consisting of a single edge~$xy$ \emph{trivial}. 
    \item An \emph{${(X, Y)}$-path} is a directed path whose initial vertex is contained in~$X$, whose terminal vertex is contained in~$Y$ and which is internally disjoint from~${X \cup Y}$. 
    \item An \emph{${(x, Y)}$-path} is a directed path whose initial vertex is~$x$, whose terminal vertex is contained in~$Y$ and which is internally disjoint from~${Y}$. 
    \item An \emph{${(X, y)}$-path} is a directed path whose initial vertex is contained in~$S$, whose terminal vertex is~$y$ and which is internally disjoint from~${X}$. 
\end{itemize}

\noindent 
We also introduce the following notation for sets of paths. 

\begin{itemize}
    \item An \emph{${(x,y)}$-path-system} is a set of ${(x,y)}$-paths, which are pairwise internally disjoint. 
    \item An \emph{${(X,Y)}$-path-system} is a set of ${(X,Y)}$-paths, which are pairwise disjoint. 
    \item An \emph{${(x, Y)}$-path-system} is a set of ${(x, Y)}$-paths, which are pairwise disjoint but for~$x$.
    \item An \emph{${(X, y)}$-path-system} is a set of ${(X, y)}$-paths, which are pairwise disjoint but for~$y$. 
\end{itemize}

\noindent 
For a set~$\mathcal{P}$ of paths, 
we write~${V^{-}(\mathcal{P})}$ and~${E^{-}(\mathcal{P})}$ for the set of the initial vertices, and initial edges, respectively, of the paths in~$\mathcal{P}$. 
Similarly, we write~${V^{+}(\mathcal{P})}$ and~${E^{+}(\mathcal{P})}$ for the set of the terminal vertices, and terminal edges, respectively, of the paths in~$\mathcal{P}$. 

\begin{itemize}
    \item An \emph{${(x,y)}$-separation} is a set~${S \subseteq V(D) \setminus \{ x,y \}}$ that meets every ${(x,y)}$-path. 
    \item An \emph{${(X,Y)}$-separation} is a set~${S \subseteq V(D)}$ of vertices that meets every ${(X,Y)}$-path. 
\end{itemize}

A set of paths~$\mathcal{P}$ and a vertex set~$S$,
are \emph{orthogonal} to each other if~${|V(P) \cap S| = 1}$ for all~${P \in \mathcal{P}}$ and~${S \subseteq V(\mathcal{P})}$. 
We write~${\mathcal{P} \orthogonal S}$ if~$\mathcal{P}$ and~$S$ are orthogonal. 

Let~${x,y \in V(D)}$ with~${xy \notin E(D)}$ and let~${X,Y \subseteq V(D)}$.
\begin{itemize}
    \item Let~$\mathcal{P}$ be an ${(x,y)}$-path-system. 
        We call~$\mathcal{P}$ an \emph{Erd\H{o}s-Menger ${(x,y)}$-path-system} if there exists an ${(x,y)}$-separation~$S$ with~${\mathcal{P} \orthogonal S}$. 
        Similarily, an ${(x,y)}$-separation~$S$ is called an \emph{Erd\H{o}s-Menger ${(x,y)}$-separation} if there is an ${(x,y)}$-path-system~$\mathcal{P}$ with~${\mathcal{P} \orthogonal S}$.
        We denote the set of Erd\H{o}s-Menger ${(x,y)}$-path-systems by~${\mathfrak{P}_D(x,y)}$ 
        and the set of Erd\H{o}s-Menger ${(x,y)}$-separations by~${\mathfrak{S}_D(x,y)}$. 
    \item Let~$\mathcal{P}$ be an ${(X,Y)}$-path-system. 
        We call~$\mathcal{P}$ an \emph{Erd\H{o}s-Menger ${(X,Y)}$-path-system} if there exists an ${(X,Y)}$-separation~$S$ with~${\mathcal{P} \orthogonal S}$. 
        An ${(X,Y)}$-separation~$S$ is called an \emph{Erd\H{o}s-Menger ${(X,Y)}$-separation} if there is an ${(X,Y)}$-path-system~$\mathcal{P}$ with~${\mathcal{P} \orthogonal S}$.
        We denote the set of Erd\H{o}s-Menger ${(X,Y)}$-path-systems by~${\mathfrak{P}_D(X,Y)}$ 
        and the set of Erd\H{o}s-Menger ${(X,Y)}$-separations by~${\mathfrak{S}_D(X,Y)}$.
\end{itemize}

Often, for a rooted~$D$ and~${v \in V(D)-r}$, we will want to consider Erd\H{o}s-Menger ${(r,v)}$-path-systems and -separations. 
If~$rv$ is an edge, then these are not defined above, and instead we will want to consider such path-systems and separations in~${D - rv}$. 
Hence, both to simplify and unify notation, we will use~${\mathfrak{P}_D(v)}$ and~${\mathfrak{S}_D(v)}$ as an abbreviation for~${\mathfrak{P}_{D-rv}(r,v)}$ and~${\mathfrak{S}_{D-rv}(r,v)}$ respectively. 

We define a partial order on~${\mathfrak{S}_D(X,Y)}$ as follows:
We write ${S \unlhd T}$ if~$S$ is an ${(X,T)}$-separation. 

\begin{lem}
    \label{lem:completelattice}
    \cite{attila-lattice}
    $\mathfrak{S}_D(X,Y)$ forms a complete lattice under~$\unlhd$.
\end{lem}

We write~${\mathcal{G}_D(v)}$ for the set of those subsets $I$ of the ingoing edges of $v$ for which there is an internally disjoint path-system from $r$ with last edges $I$. More formally: 
\[\mathcal{G}_D(v):=\{I \subseteq \ingoing_D(v):\, \exists\,  (r,v)\text{-path-system } \mathcal{P} \text{ with } I=E^{+}(\mathcal{P}) \}.\]


In terms of this notation, we can define a \emph{flame} as a rooted digraph $F$ in which ${\ingoing_F(v) \in \mathcal{G}_F(v)}$ for every~${v \in V(F)}$.

\vspace{0.2cm}

We will need an adaption of the following theorem due to Pym. 

\begin{thm}[Pym]
    \label{thm:pym}
    \cite{pym}
    Let~$D$ be a digraph, let~${X,Y \subseteq V(D)}$, and let~$\mathcal{P}$ and~$\mathcal{Q}$ be~${(X,Y)}$-path-systems in~$D$. 
    Then there is an ${(X,Y)}$-path-system~$\mathcal{R}$ 
    for which~${V^-(\mathcal{R}) \supseteq V^-(\mathcal{P})}$ and~${V^+(\mathcal{R}) \supseteq V^+(\mathcal{Q})}$. 
\end{thm}

\begin{corollary}
    \label{cor:pym}
    Let~$D$ be a digraph, let~${X \subseteq V(D)}$, $y \in V(D) \setminus X$, and let~$\mathcal{P}$ and~$\mathcal{Q}$ be~${(X, y)}$-path-systems in~$D$. 
    Then there is an ${(X, y)}$-path-system~$\mathcal{R}$ of 
    with the properties~${V^-(\mathcal{R}) \supseteq V^-(\mathcal{P})}$ and~${E^+(\mathcal{R}) \supseteq E^+(\mathcal{Q})}$. 
\end{corollary}

\begin{proof}
    After subdividing the ingoing edges of~$y$, calling the set of new vertices~$Y$, we can apply Theorem~\ref{thm:pym} to obtain the ${(X,Y)}$-path-system, which can be similarly translated back to an ${(X,y)}$-path-system with the desired properties. 
\end{proof}

The following Lemma is a corollary of Pym's Theorem, and shows that if~${I \in \mathcal{G}_{D-rv}(v)}$, then the witness for it can be chosen from~${\mathfrak{P}_{D}(v)}$.

\begin{lem}
    \label{lem:covering-independent0}
    Let~$D$ be a rooted digraph and let~${v \in V(D)-r}$. 
    For every~${I \in \mathcal{G}_{D-rv}(v)}$ 
    and~${S \in \mathfrak{S}_{D}(v)}$ 
    there is an~${\mathcal{R}\in \mathfrak{P}_{D}(v)}$ orthogonal to~$S$ 
    with~${I \subseteq E^+(\mathcal{R})}$.
\end{lem}

\begin{proof}
    By the definition of~${\mathfrak{S}_{D}(v)}$ there is a ${\mathcal{P} \in \mathfrak{P}_{D}(v)}$ orthogonal to~$S$. 
    Let~$\mathcal{Q}$ be an $(r,v)$-path-system witnessing~${I \in \mathcal{G}_{D}(v)}$. 
    Let~$\mathcal{P}'$ and~$\mathcal{Q}'$ consist of the terminal segments of the paths~$\mathcal{P}$ and~$\mathcal{Q}$ from the last common vertex with~$S$ respectively. 
    Applying Corollary~\ref{cor:pym} with~$\mathcal{P}'$ and~$\mathcal{Q}'$ results in an $(S,v)$-path-system~$\mathcal{R}'$ with ${V^-(\mathcal{R}') = S}$ and~${I \subseteq E^+(\mathcal{R}')}$. 
    The concatenation of the initial segments of the paths~$\mathcal{P}$ until~$S$ with the paths~$\mathcal{R}'$ is a desired~$\mathcal{R}$. 
\end{proof}

\vspace{0.2cm}

Lastly, we will use the standard tool of the augmenting walk method. 

\begin{lem}\label{lem:augmentingwalk}
    Let $\mathcal{P}$ be an $(X,Y)$-path-system in~$D$. 
    Then  
    \begin{itemize}
        \item either there is an~${x \in X \setminus V^{-}(\mathcal{P})}$, a ${y \in Y \setminus V^{+}(\mathcal{P})}$ and an ${(X,Y)}$-path-system~$\mathcal{P}'$ in~$D$ such that:
        \begin{itemize}
            \item ${V^{-}(\mathcal{P}') = V^{-}(\mathcal{P}) \cup \{x\}}$;
            \item ${V^{+}(\mathcal{P}') = V^{+}(\mathcal{P}) \cup \{y\}}$;
            \item $|E(\mathcal{P}) \triangle E(\mathcal{P}')| < \infty$;
        \end{itemize}
        \item or there is an $(X,Y)$-separation~$S$ such that~${\mathcal{P} \orthogonal S}$.
    \end{itemize}
\end{lem}

For more details about the augmenting walk lemma and its role in the proof of the Aharoni-Berger theorem we refer to Lemmas~3.3.2 and~3.3.3 as well as Theorem~8.4.2 of~\cite{diestel-book}.

\subsection{Largeness}
\label{subsec:largeness}
\ 

Let~$D$ be a rooted digraph and let~${v \in V(D)-r}$.
A spanning subgraph ${L \subseteq D}$ is
\emph{$v$-large} with respect to $D$ if some ${\mathcal{P} \in \mathfrak{P}_{D}(v)}$ lies in $L$ and~$L$ contains~$rv$ if~$D$ does. 
Then,~$L$ is called \emph{$D$-large} if~$L$ is $v$-large for all~${v \in V(D)-r}$. 
Note that a $D$-large subgraph in particular contains~$\outgoing_D(r)$. 
When the ambient digraph is clear from the context we will simply call~$L$ \emph{large}. 

We obtain the following corollary from Lemma~\ref{lem:covering-independent0}. 

\begin{corollary}
    \label{cor:covering-independent}
    Let~$D$ be a rooted digraph, let~${v \in V(D)-r}$ 
    and let~${L \subseteq D}$ be $v$-large. 
    Then for every~${I \in \mathcal{G}_L(v)}$ there is an $(r,v)$-path-system~$\mathcal{P}$ 
    with~${(\mathcal{P}-\{rv\}) \in \mathfrak{P}_{D}(v) \cap \mathfrak{P}_{L}(v)}$ and~${I \subseteq E^+(\mathcal{P})}$.
\end{corollary}

\begin{proof}
    By the definition of $v$-largeness, 
    there exists some~${S \in \mathfrak{S}_{D}(v)\cap \mathfrak{S}_{L}(v)}$. 
    We apply Lemma~\ref{lem:covering-independent0} with~$L$,~$v$,~${I-rv}$ and~$S$ 
    and extend the resulting path-system with the trivial path~${\{ rv \}}$ if~${rv \in I}$. 
\end{proof}

We will also need the following results about largeness from~\cite{attila-flames}. 
The first one gives us a condition for when largeness with respect to a subgraph~$G$ of~$D$ implies largeness with respect to~$D$.

\begin{lem}[\cite{attila-flames}*{Lemma~3.13}]
    \label{lem:superlarge}
    Let~$D$ be a rooted digraph and let~${G \subseteq D}$. 
    If for every~${uv \in E(D) \setminus E(G)}$ there is an~${I \in \mathcal{G}_G(v)}$ such that~${(I+uv) \notin \mathcal{G}_{G+uv}(v)}$, 
    then whenever~$H$ is $G$-large it is $D$-large as well.
\end{lem}

The second one says that to verify that~$L$ is large it is sufficient to show $v$-largeness only for vertices whose in-neighbourhoods differ in~$L$ and~$D$.

\begin{lem}[\cite{attila-flames}*{Lemma~2.2}]
    \label{lem:maintainlarge}
    Let~$D$ be a rooted digraph, let ${L\subseteq D}$ and let \linebreak 
    ${M := \{ v\in V \colon \ingoing_L(v)\subsetneq \ingoing_D(v) \}}$. 
    If~$L$ is $v$-large for all ${v \in M}$, then~$L$ is large. 
\end{lem}

\subsection{Quasi-flames}
\label{subsec:quasi-flames} 
\ 

We extend the notion of a quasi-flame as defined in~\cite{attila-flames}. 
Let~${D}$ be a rooted digraph, let~${G\subseteq D}$ and let $v \in V(D)-r$. 
We say that~$D$ has the \emph{$G$-quasi-flame property at $v$} if~${I \in \mathcal{G}_{D}(v)}$ whenever ${I \subseteq \ingoing_D(v)}$ with ${I \setminus \ingoing_G(v)}$ finite. 
We call~$D$ a \emph{$G$-quasi-flame} if it has the $G$-quasi-flame property at each vertex. 
For example the digraph in Figure~\ref{f:ex-flame} is a $G$-quasi-flame for every spanning subgraph~$G$ with countably many edges.

Our aim is to reduce the main theorem to the case where~$D$ is an $F$-quasi-flame. 

\begin{cor}
    \label{t:large}
    Let~$F$ be a vertex-flame in~$D$. 
    Then there is an $F$-quasi-flame ${Z \subseteq D}$ such that whenever ${L \subseteq Z}$ is $Z$-large, then~$L$ is also $D$-large.
\end{cor}

\begin{proof}
    Observe that~$F$ itself is an $F$-quasi-flame and the union of an $\subseteq$-increasing non-empty chain of $F$-quasi-flames in~$D$ is an $F$-quasi-flame. 
    Thus by Zorn's Lemma we may take a $\subseteq $-maximal $F$-quasi-flame~$Z$ in~$D$. 
    To show that~$Z$ has the desired properties, we apply Lemma~\ref{lem:superlarge}. 
    Assume that~${uv \in E(D) \setminus E(Z)}$. 
    By the maximality of~$Z$, ${Z+uv}$ is not an $F$-quasi-flame. 
    The only possible reason for this is the existence of an~${I \in \mathcal{G}_Z(v)}$ with~${I \supseteq \ingoing_F(v)}$ 
    where~${I \setminus \ingoing_A(v)}$ finite and such that ${(I+uv) \notin \mathcal{G}_{Z+uv}(v)}$. 
    Thus we are done by applying Lemma~\ref{lem:superlarge}.
 \end{proof}

For a countable $G$-quasi-flame~$D$ and a vertex~${v \in V(D)}$,  
the following key lemma allows us to find a set ${I^* \in \mathcal{G}_D(v)}$ 
such that the restriction of~$D$ to~$I^*$ at~$v$ is still a $G$-quasi-flame. 

\begin{restatable}{lemma}{key}\label{lem:key}
    Let $D$ be a countable $G$-quasi-flame for some ${G \subseteq D}$ and let~${v \in V(D)-r}$. 
    Then there is an ${I^* \in \mathcal{G}_D(v)}$ such that~${D \upharpoonright_v I^*}$ is a $G$-quasi-flame.
\end{restatable}

We will prove this key lemma in Section~\ref{sec:key-lemma} after laying the necessary ground work in Section~\ref{sec:incompressibility}.

\section{Proof of the main theorem}
\label{sec:proof}

In this section we prove our main result Theorem~\ref{thm:extend}. 
To do so we assume Lemma~\ref{lem:key}, whose proof is postponed to Section~\ref{sec:key-lemma}.

\extend*

\begin{proof}

Note that the union of any flame~$F$ in a given rooted digraph~$D$ with~$\outgoing_D(r)$ is still a flame. 
Hence we may assume that~$F$ contains~$\outgoing_D(r)$. 
By Corollary~\ref{t:large} we can also assume without loss of generality that~$D$ is an $F$-quasi-flame. 

Let ${\{v_n \colon n \in \mathbb{N} \}}$ be an enumeration of~${V(D)-r}$.
Let~${G_{-1} := F}$ and~${L_{-1} := D}$. 
We will recursively construct 
\begin{itemize}
    \item a sequence ${(\mathcal{P}_n \colon n \in \mathbb{N})}$ of Erd\H{o}s-Menger path-systems ${\mathcal{P}_n \in \mathfrak{P}_{D}(v_{n})}$;
    \item two sequences $(G_n \colon n \in \mathbb{N})$ and $(L_n \colon n \in \mathbb{N})$ of spanning subdigraphs of $D$; 
\end{itemize}
such that for all~${n \in \mathbb{N}}$:
\begin{enumerate}
    [label=(\arabic*)]
    \item\label{item0} $\mathcal{P}_n\in \mathfrak{P}_{L_{n-1}}(v_{n})\cap \mathfrak{P}_{D}(v_{n})$ with $\ingoing_{G_{n-1}}(v_n) \subseteq E^+(\mathcal{P}_n)\cup \{ rv_n \}$;
    \item\label{item1} $G_{n}$ is obtained by adding the edges of the paths in $\mathcal{P}_n$ to $G_{n-1}$; 
    \item\label{item2} $L_{n} = L_{n-1} \upharpoonright_{v_n} \ingoing_{G_n}(v_n)$; 
    \item\label{item4} $L_n$ is a $G_n$-quasi-flame.
\end{enumerate}

By~\ref{item0},~\ref{item1} and~\ref{item2}, $(G_n \colon n \in \mathbb{N})$ is $\subseteq$-increasing, $(L_n \colon n \in \mathbb{N})$ is $\subseteq$-decreasing and~${G_n \subseteq L_n}$ for all~${n \in \mathbb{N}}$. 
Moreover, 
\[{
    \ingoing_{G_{n}}(v_n) \setminus \{rv_n\} = \ingoing_{L_{n}}(v_n) \setminus \{rv_n\} = E^+(\mathcal{P}_n)}.
\] 
Hence ${\bigcup_{n \in \mathbb{N}} G_n = \bigcap_{n \in \mathbb{N}} L_n =: F^*}$, and~${F \subseteq F^*}$. 
Furthermore, ${\mathcal{P}_n \in \mathfrak{P}_{D}(v_{n})}$ ensures that~$F^*$ is $v_n$-large and ${\ingoing_{F^*}(v_n) \in \mathcal{G}_{F^*}(v_n)}$. 
Combining these we conclude that~$F^*$ is a large flame in~$D$ extending~$F$, as desired. 

Suppose that $G_{n-1}$ and~$L_{n-1}$ are defined for some~${n \in \mathbb{N}}$ and~$\mathcal{P}_i$ is defined for all~$i$ with~${0 \leq i < n}$. 
By Lemma~\ref{lem:key} applied to~$L_{n-1}$, $G_{n-1}$ and~$v_n$ there is an ${I^* \in \mathcal{G}_{L_{n-1}}(v_n)}$ such that ${L_{n-1} \upharpoonright_{v_n} I^*}$ is a $G_{n-1}$-quasi-flame. 

Since ${\mathcal{P}_i \in \mathfrak{P}_{L_{n-1}}(v_{i})}$ 
for all~${i < n}$, properties~\ref{item1} and~\ref{item2} imply that~$L_{n-1}$ satisfies the conditions of Lemma~\ref{lem:maintainlarge} and hence~$L_{n-1}$ is large. 
Thus, we may apply Corollary~\ref{cor:covering-independent} to~$L_{n-1}$, $v_n$ and~${I^* \in \mathcal{G}_{L_{n-1}}(v_n)}$ to find
an Erd\H{o}s-Menger path-system ${\mathcal{P}_n \in \mathfrak{P}_{D}(v_n) \cap \mathfrak{P}_{L_{n-1}}(v_n)}$ such that~${I^* \subseteq E^+(\mathcal{P}_n) \cup \{rv_n\}}$. 
We define~$G_n$ and~$L_n$ according to~\ref{item1} and~\ref{item2}, 
i.e., ${G_{n} := G_{n-1} \cup \bigcup\mathcal{P}_n}$ and~${L_{n} := L_{n-1} \upharpoonright_{v_n} \ingoing_{G_n}(v_n)}$. 
We claim that~$G_{n}$, $L_{n}$ and~$\mathcal{P}_{n}$ satisfy properties~{\ref{item0}--\ref{item4}}. 

Clearly, properties~\ref{item1} and~\ref{item2} are satisfied. 
By construction, ${\bigcup \mathcal{P}_n \subseteq L_{n-1}}$ and \linebreak 
${\ingoing_{G_{n-1}}(v_n) \subseteq I^* \subseteq E^+(\mathcal{P}_n)\cup \{ rv \}}$. 
Finally, we need to show that~$L_{n}$ is a $G_{n}$-quasi-flame. 

Note that by construction~${\ingoing_{L_{n}}(v_n) \in \mathcal{G}_{L_{n}}(v_n)}$, as guaranteed by~$\mathcal{P}_n$. 
Hence we only need to check that the $G_n$-quasi-flame property holds at~${v_m \neq v_n}$. 
However, ${L_{n-1} \upharpoonright_{v_n} I^*}$ is a $G_{n-1}$-quasi-flame and hence, since $I^* \subseteq \ingoing_{G_n}(v_n)$, it follows that $L_n$ satisfies the $G_{n-1}$-quasi-flame property at every~${v_m \neq v_n}$. 
Furthermore, since~${\ingoing_{G_{n}}(v_m)}$ has at most one new edge compared to~${\ingoing_{G_{n-1}}(v_m)}$, it follows that~$L_{n}$ satisfies the $G_{n}$-quasi-flame property at every~${v_m \neq v_n}$. 
\end{proof}

\section{Incompressibility}\label{sec:incompressibility}

Incompressibility will be a key concept in the proof of Lemma~\ref{lem:key} but its scope of potential applications is not restricted to flames. 
For example it played an important role in the arborescence packing result~\cite{arbpacking} of the third author. 

We will be interested in when we can find an ${(X,Y)}$-path-system which covers the vertices of~$X$, in which case we say that~$X$ is joinable to~$Y$. 
If~$X$ and~$Y$ are both infinite, then clearly this cannot happen if there is some finite ${(X,Y)}$-separation in~$D$. 
However, even when the two sets are infinitely connected, it can happen that~$X$ is not joinable to~$Y$.

One obstruction to this would be an ${(X,Y)}$-separation~$S$ which is `smaller than~$X$' in the sense that it is the `same size' as a proper subset of~$X$. 
For finite sets~$U$ and~$W$ it is clear that the equality~${|U| = |W|}$ is equivalent to the condition that every injective map from~$U$ to~$W$ is also a bijection. 
For countably infinite sets this is no longer true, since~$U$ can be injected, or `compressed', into an arbitrary infinite subset of~$W$. 
However if we place restrictions on the types of injections we allow, then there may be sets~$U$ and~$W$ where only bijections are possible, and so~$U$ is `incompressibile' to~$W$, respresenting a sort of tightness between~$U$ and~$W$ with respect to the set of injections.

In our particular case we are considering the injections which arise from ${(U,W)}$-path-systems which cover the vertices in~$U$. 
Each such system defines in a natural way an injection from~$U$ to~$W$, and for certain pairs of infinite sets, the only injections which arise in this manner will be bijections. 
In this case we will say that~$U$ is incompressible to~$W$. 

Then, an $(X,Y)$-separation~$S$ can be the `same size' as a proper subset~$X'$ of~$X$, if~$X'$ is incompressible to~$S$. 
It is then easy to see this is an obstruction to~$X$ being joinable to~$Y$, as any ${(X,Y)}$-path-system witnessing this would contain an ${(X',S)}$-path-system, which by the incompressability of~$X'$ to~$S$ would cover~$S$, and so separate~$X$ from~$Y$. 

We will show that, if there is some $\subseteq$-maximal proper subset~$X'$ of~$X$ which is joinable to~$Y$, 
then this will in fact by witnessed by such a `tight' ${(X,Y)}$-separation~$S$, that is, where $X'$ is incompressible to~$S$.

\vspace{0.2cm}

\begin{definition}
    \label{def:incompressability}
    
    Let~$D$ be a digraph and let~${X,Y \subseteq V(D)}$. 
    \begin{itemize}
        \item We say~$X$ is \emph{joinable} to~$Y$ in~$D$ if there is an ${(X,Y)}$-path-system~$\mathcal{P}$ in~$D$ with~${V^{-}(\mathcal{P}) = X}$. We denote the set of all such path-systems by~$\mathfrak{J}_{D}(X,Y)$. 
        \item Suppose~$X$ is joinable to~$Y$. 
            We say that~$X$ is \emph{incompressible} to~$Y$ in~$D$ if for every~${\mathcal{P} \in \mathfrak{J}_{D}(X,Y)}$ we have~${V^{+}(\mathcal{P}) = Y}$. 
    \end{itemize}
    For examples of these terms, see Figure~\ref{f:ex-incomp}. 
\end{definition}

    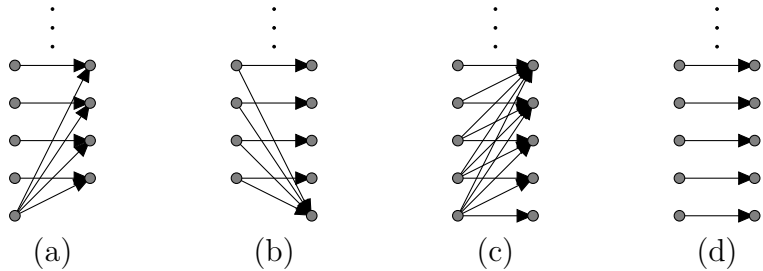
\begin{figure}[htbp]
        \centering
        \begin{tikzpicture}
            [scale=0.5]
            \tikzset{vertex/.style = {circle, draw, fill=black!50, inner sep=0pt, minimum width=4pt}}
            \tikzset{dot/.style = {circle, draw, fill=black, inner sep=0pt, minimum width=1pt}}
            \tikzset{arrow/.style = {-triangle 45}}
            
            \pgfmathtruncatemacro{\n}{4};
            \pgfmathtruncatemacro{\nminus}{\n-1};
        
            \foreach \i in {0,...,\n} {
                \coordinate (v\i) at (0,\i) {};
                \coordinate (w\i) at (2,\i) {};
            }
            
            \foreach \i in {1,...,\n} {
                \draw [-triangle 45] (v\i) -- (w\i) {};
                \draw [-triangle 45] (v0) -- (w\i) {};
            }
            
            \foreach \i in {1,...,\n} {
                \node[vertex] at (v\i) {};
                \node[vertex] at (w\i) {};
            }
            \node[vertex] at (v0) {};
            
            \foreach \i in {1,...,3} {
                \node[dot] at (1,\n+0.5*\i) {};
            }
            \node at (1,-1) {(a)};
        \end{tikzpicture}
        \qquad
        \qquad
        \begin{tikzpicture}
            [scale=0.5]
            \tikzset{vertex/.style = {circle, draw, fill=black!50, inner sep=0pt, minimum width=4pt}}
            \tikzset{dot/.style = {circle, draw, fill=black, inner sep=0pt, minimum width=1pt}}
            \tikzset{arrow/.style = {-triangle 45}}
            
            \pgfmathtruncatemacro{\n}{4};
            \pgfmathtruncatemacro{\nminus}{\n-1};
        
            \foreach \i in {0,...,\n} {
                \coordinate (v\i) at (0,\i) {};
                \coordinate (w\i) at (2,\i) {};
            }
            
            \foreach \i in {1,...,\n} {
                \draw [-triangle 45] (v\i) -- (w\i) {};
                \draw [-triangle 45] (v\i) -- (w0) {};
            }
            
            \foreach \i in {1,...,\n} {
                \node[vertex] at (v\i) {};
                \node[vertex] at (w\i) {};
            }
            \node[vertex] at (w0) {};
            
            \foreach \i in {1,...,3} {
                \node[dot] at (1,\n+0.5*\i) {};
            }
            \node at (1,-1) {(b)};
        \end{tikzpicture}
        \qquad
        \qquad
        \begin{tikzpicture}
            [scale=0.5]
            \tikzset{vertex/.style = {circle, draw, fill=black!50, inner sep=0pt, minimum width=4pt}}
            \tikzset{dot/.style = {circle, draw, fill=black, inner sep=0pt, minimum width=1pt}}
            \tikzset{arrow/.style = {-triangle 45}}
            
            \pgfmathtruncatemacro{\n}{4};
            \pgfmathtruncatemacro{\nminus}{\n-1};
        
            \foreach \i in {0,...,\n} {
                \coordinate (v\i) at (0,\i) {};
                \coordinate (w\i) at (2,\i) {};
            }
            
            \foreach \i in {0,...,\n} {
                \foreach \j in {\i,...,\n} {
                    \draw [-triangle 45] (v\i) -- (w\j) {};
                }
            }
            
            \foreach \i in {0,...,\n} {
                \node[vertex] at (v\i) {};
                \node[vertex] at (w\i) {};
            }
            
            \foreach \i in {1,...,3} {
                \node[dot] at (1,\n+0.5*\i) {};
            }
            \node at (1,-1) {(c)};
        \end{tikzpicture}
        \qquad
        \qquad
        \begin{tikzpicture}
            [scale=0.5]
            \tikzset{vertex/.style = {circle, draw, fill=black!50, inner sep=0pt, minimum width=4pt}}
            \tikzset{dot/.style = {circle, draw, fill=black, inner sep=0pt, minimum width=1pt}}
            \tikzset{arrow/.style = {-triangle 45}}
            
            \pgfmathtruncatemacro{\n}{4};
            \pgfmathtruncatemacro{\nminus}{\n-1};
        
            \foreach \i in {0,...,\n} {
                \coordinate (v\i) at (0,\i) {};
                \coordinate (w\i) at (2,\i) {};
            }
            
            \foreach \i in {0,...,\n} {
                \draw [-triangle 45] (v\i) -- (w\i) {};
            }
            
            \foreach \i in {0,...,\n} {
                \node[vertex] at (v\i) {};
                \node[vertex] at (w\i) {};
            }
            
            \foreach \i in {1,...,3} {
                \node[dot] at (1,\n+0.5*\i) {};
            }
            \node at (1,-1) {(d)};
        \end{tikzpicture}
        \label{f:ex-incomp}
        \caption{Examples for the terms in Definitions~\ref{def:incompressability} and~\ref{def:finitelyextendable}. 
            The set~$X$ in each example is the set of sources, the set~$Y$ the set of sinks. 
            In~(a), $X$ is not joinable to~$Y$. 
            In~(b) and~(c), $X$ is joinable to~$Y$ but not incompressible. In fact in~(b) no~$(X,Y)$-path system covers~$Y$, while in~(c) there exists a perfect matching. 
            In~(d), $X$ is incompressible to~$Y$. 
            Note that in~(a), every subset~$X'$ of~$X$ for which~${X \setminus X'}$ is infinite is finitely extenable.}
    \end{figure}

Let~$D$,~$X$ and~$Y$ be fixed as above. 
For technical reasons we assume that the vertices in~$X$ have no ingoing edges and the vertices in~$Y$ have no outgoing edges. 
Those edges will never be contained in any~${(X,Y)}$-path and hence to not play any role in any of the statements.

\begin{lem}
    \label{lem:incomp}
    Let~${x \in X}$ and assume that 
    ${(X-x)}$ is joinable to~$Y$ but~$X$ is not. 
    Then there is an ${S \in \mathfrak{S}_D(X,Y)}$ such that~${(X-x)}$ is incompressible to~$S$ in~$D$. 
\end{lem}

\begin{proof}
The Aharoni-Berger theorem \cite{infinite-menger} ensures that~${\mathfrak{S}_D(X,Y) \neq \emptyset}$. 
    Let~$S$ be the $\unlhd$-smallest element of~${\mathfrak{S}_D(X,Y)}$, which exists by Lemma~\ref{lem:completelattice}, 
    and let~${\mathcal{P} \in \mathfrak{J}(X-x,S)}$ be arbitrary. 
    Since there is a~${\mathcal{P'} \in \mathfrak{P}_D(X,Y)}$ 
    orthogonal to~$S$, 
    if~$X$ is joinable to~$S$ 
    then we can extend an $(X,S)$-path-system using the terminal segments of~$\mathcal{P}'$ starting at~$S$ to one witnessing 
    that~$X$ is joinable to~$Y$. 
    Hence, $X$ is not joinable to~$S$. 
    
    The vertices in~$X$ have no ingoing edges by assumption, and hence none of the paths in~$\mathcal{P}$ goes through~$x$. 
    Applying Lemma~\ref{lem:augmentingwalk} to~$X$,~$S$ and~$\mathcal{P}$ we see that the first case of the lemma would allow us to conclude 
    that~$X$ is joinable to~$Y$ 
    and hence there is some ${(X,S)}$-separation~$T$ such that~${\mathcal{P} \orthogonal T}$. 
    Since~$T$ is also an ${(X,Y)}$-separation, it follows that~${T \in \mathfrak{S}_D(X,Y)}$ and hence by $\unlhd$-minimality of~$S$, we have~${S = T}$ and so~${V^{+}(\mathcal{P}) = S}$.
\end{proof}

\begin{lem}
    \label{lem:incomp-as-many-new}
    Suppose that~$X$ is joinable to~$Y$ and~$X'$ is joinable to~$Y'$ 
    for some~${X' \subseteq X}$ and~${Y' \subseteq Y}$ with~${ \left|X\setminus X'\right|<\infty}$. 
    Then there is some~${Y'' \subseteq Y}$ 
    such that~$X$ is joinable to~${Y''}$ 
    and~${\left|Y'' \setminus Y'\right| \leq \left|X \setminus X'\right|}$.
\end{lem}

\begin{proof}
    It is enough to prove the case where~${\left|X\setminus X'\right| = 1}$, since the general case then follows by applying this recursively. 
    
    Let~${\mathcal{P} \in \mathfrak{J}(X',Y')}$. 
    We apply Lemma~\ref{lem:augmentingwalk} to~$X$,~$Y$ and~$\mathcal{P}$. 
    If the first case of the lemma holds, then we augment to an ${(X,Y)}$-path-system~$\mathcal{P'}$ which witnesses 
    that~$X$ is joinable to~$Y''$ 
    such that~${Y'' = Y' \cup \{y\}}$ for some~${y \in Y\setminus Y'}$ and hence~${|Y'' \setminus Y'| = 1 = |X \setminus X'|}$. 
    Otherwise, there is some $(X,Y)$-separation~$S$ with~${S \orthogonal \mathcal{P}}$. 
    However, 
    since~$X$ is joinable to~$Y$ 
    and~$S$ is an $(X,Y)$-separation, it follows 
    that~$X$ is joinable to~$S$. 
    But then by extending an ${(X,S)}$-path-system 
    witnessing that~$X$ is joinable to~$S$ 
    with the terminal segments of~$\mathcal{P}$ from~$S$ to~$Y'$ we get an ${(X,Y)}$-path-system witnessing 
    that~$X$ is joinable to~$Y'$. 
\end{proof}

\begin{lem}
    \label{lem:incomp-finite-difference}
    Let~$D'$ be a subgraph of~$D$ obtained by deleting finitely many vertices of~$D$ and then deleting finitely many edges from the resulting digraph. If~$X$ is incompressible to~$Y$ in~${D'}$, then~$X$ was already incompressible to~$Y$ in~$D$.
\end{lem}

\begin{proof}
    Suppose for a contradiction that there is some~${\mathcal{P} \in \mathfrak{J}_D(X,Y)}$ for which~${Y^* := V^{+}(\mathcal{P})}$ is a proper subset of~$Y$. 
    Since the paths in~$\mathcal{P}$ are vertex disjoint, 
    all but finitely many of the paths in~$\mathcal{P}$ are contained in~${D'}$. 
    Let~${\mathcal{P}'= \{ P \in \mathcal{P} \colon P \subseteq D' \}}$. 
    Then there is some~${Y' \subseteq Y^*}$ and~${X' \subseteq X}$ such that ${\mathcal{P}' \in \mathfrak{J}_{D'}(X',Y')}$ and 
    \[
        |Y^* \setminus Y'| = |X \setminus X'| = |\mathcal{P} \setminus \mathcal{P}'| =: k.
    \]
    Note in particular that~${|Y \setminus Y'| > |Y^* \setminus Y'|=k}$. 
    Since, by assumption~$X$ is joinable to~$Y$ in~$D'$, 
    by Lemma~\ref{lem:incomp-as-many-new} applied to~$X$,~$X'$,~$Y$ and~$Y'$ there is some~${Y'' \subseteq Y}$ 
    such that~$X$ is joinable to~$Y''$ in~$D'$ 
    and 
    \[
        |Y'' \setminus Y'| \leq |X \setminus X'| = k.
    \]
    However, then 
    \[
        |Y \setminus Y''| \geq |Y \setminus Y'| - |Y'' \setminus Y'| > 0,
    \]
    contradicting the fact that~$X$ is incompressible to~$Y$ in~${D'}$. 
\end{proof}

\begin{lem}\label{lem:incomp-remains}
    If there exists~${x \in X}$ and~${y \in Y}$ such that~${X - x}$ is incompressible to~${Y - y}$ in~$D$ 
    and~$X$ is joinable to~$Y$ in~$D$, 
    then~$X$ is incompressible to~$Y$ in~$D$.
\end{lem}

\begin{proof}
    Assume, for a contradiction, that there is some ${\mathcal{P} \in \mathfrak{J}(X,Y)}$ with~${Y \setminus V^{+}(\mathcal{P}) \neq \emptyset}$. 
    Let~${P \in \mathcal{P}}$ be the path starting in~$x$. 
    Since~${X - x}$ is incompressible to~${Y-y}$ there is some~${Q \in \mathcal{P}}$ with~${P \neq Q}$ which ends in~$y$. 
    Let~${\mathcal{P}' = \mathcal{P} \setminus \{ P, Q \}}$. 

    Then~${\mathcal{P}' \in \mathfrak{J}(X',Y')}$ for some~${X' \subseteq X-x}$ and~${Y' \subseteq Y-y}$ with~${|(X-x) \setminus X'| = 1}$ 
    and ${|(Y-y) \setminus Y'| > 1}$. 
    However, then by Lemma~\ref{lem:incomp-as-many-new} applied to~${(X-x)}$,~$X'$,~$(Y-y)$ and~$Y'$ there is some ${Y'' \subsetneq (Y-y)}$
    such that~${(X-x)}$ is joinable to~$Y''$ contradicting the fact that~${X - x}$ is incompressible to~${Y - y}$ in~$D$.
\end{proof}

\begin{definition}
    \label{def:finitelyextendable}
    Let~$D$ be a digraph and let~${X,Y \subseteq V(D)}$. 
    A subset~${X' \subseteq X}$ is \emph{$(X,Y)$-finitely extendable} in~$D$ 
    if any~${O \subseteq X}$ containing~$X'$ with~${|O \setminus X'| < \aleph_0}$ is joinable to~$Y$.
\end{definition}
A natural question is whether this implies 
that~$O$ is joinable to~$Y$
for some~${ X'\subseteq O \subseteq X}$ with~${|O \setminus X'| = \aleph_0}$. 
We will actually need a stronger statement, given a countable collection of infinite subsets of~${X \setminus X'}$ we want to find such an~$O$ which meets all of them. 

To do so we will need to establish the following lemma first.

\begin{lem}
    \label{lem:incomp-delete-one}
    Let ${X' \subseteq X}$ such that~$X'$ is $(X,Y)$-finitely extendable. Suppose that for some finite~${U\subseteq V(D)\setminus X'}$ the set~$X'$ is joinable to~${Y \setminus U}$ in~${D-U}$. 
    Then~$U$ can be extended by at most~${\left| U \right|}$ many new vertices from~${X \setminus X'}$ to a set~$W$ such that~$X'$ is ${(X \setminus W , Y \setminus W)}$-finitely extendable in~${D - W}$. 
\end{lem}

\begin{proof}
    It is enough to prove the special case where~${U = \{v\}}$ for some~${v \in V(D)}$, the general case follows by applying it to the vertices in~$U$ one by one recursively. 
    
    If~${v \in X}$, then ${W := \{ v \}}$ is suitable. 
    Indeed, since the vertices in~$X$ have no ingoing edges, the deletion of~$v$ cannot ruin anything in this case. 
    So, let us suppose that~${v \notin X}$. 
    We may assume that there is a set 
    ${\{ x_i \colon 0 \leq i \leq \ell\} \subseteq X \setminus X'}$ 
    for which
    ${X' \cup \{ x_i \colon 0 \leq i \leq \ell \}}$ 
    is not joinable to~${Y - v}$ in~${D - v}$ 
    but~${X' \cup \{ x_i \colon 0 \leq i < \ell \}}$ 
    is joinable to~${Y - v}$ in~${D - v}$. 
    Without loss of generality we can assume that~${\ell = 0}$ since otherwise we replace~$X'$ 
    with~${X' \cup \{ x_i \colon 0 \leq i < \ell \}}$. 
    In the light of this we write simply~$x$ instead of~$x_0$ and show that~${W := \{v, x \}}$ satisfies the conclusion of the lemma.

    By Lemma~\ref{lem:incomp} there is an~${S \in \mathfrak{S}_{D-v}(X',Y-v)}$ such that~${X'}$ is incompressible to~$S$ in~${D - v}$. 
    However, then by Lemma~\ref{lem:incomp-finite-difference}, ${X'}$ is incompressible to~$S$ in~$D$ as well. 
    Note that~${S + v}$ separates~$Y$ from~$X'$ in~$D$ 
    and by assumption~${X' + x}$ is joinable to~$Y$ in~$D$. 
    Hence,~${X' + x}$ is joinable to~${S + v}$ in~$D$. 
    Then, by Lemma~\ref{lem:incomp-remains} applied to~${x \in X' + x}$ and~${v \in S + v}$ we can conclude that~${X' + x}$ is also incompressible to~${S + v}$ in~$D$. 
    
    Next, we show 
    that~$O$ is joinable to~$Y-v$ in~${D-v}$ 
    for any ${X' \subseteq O \subseteq X-x}$ 
    for which~${|O \setminus X'| < \aleph_0}$. 
    Let an arbitrary such~$O$ be fixed (see Figure~\ref{f:joinOtoY}). 
    Since~$X'$ is $(X,Y)$-finitely extendable in~$D$, there is some set of paths 
    ${\mathcal{P} \in \mathfrak{J}_D(O+x,Y)}$.
    \begin{figure}[htbp]
        \centering
        \begin{tikzpicture}
            
            \draw (3.2,3) rectangle (4.4,-1.4);
            \node at (3.8,3.2) {$Y$};
            
            \draw (-2.4,3) node (v1) {} rectangle (-1,-1.4);
            \draw (0.6,3) rectangle (1.4,1);
            
            \draw (v1) rectangle (-1,1);
            \node at (1,3.2) {$S+v$};
            
            \node at (-1.8,3.2) {$O+x$};
            \node (v3) at (1,2.8) {$v$};
            \node (v2) at (-1.8,2.6) {$x$};
            \node at (-3.1,2) {$X'+x$};
            \node at (1.3,-1.7) {${\mathcal{P} \setminus \mathcal {P}'}$};
            \node at (-0.4,2.4) {$\mathcal{Q}$};
        
            \draw [-triangle 60, dashed] plot[smooth, tension=.7] coordinates {(1.2,2.8) (2.2,2.4) (3.6,2.6)};
            
            \draw [-triangle 60] plot[smooth, tension=.7] coordinates {(-1.8,2.2) (-0.4,2) (1,2.2)};
            \draw [-triangle 60] plot[smooth, tension=.7] coordinates {(-1.8,1.8) (-0.2,1.6) (1,1.8)};
            \draw [-triangle 60] plot[smooth, tension=.7] coordinates {(-1.8,1.4) (-0.2,1.2) (1,1.4)};
        
            \draw [-triangle 60, dashed] plot[smooth, tension=.7] coordinates {(1,2.2) (1.8,2.2) (2.4,2) (3.6,1.8)};
            \draw [-triangle 60, dashed] plot[smooth, tension=.7] coordinates {(1,1.8) (2.2,1.6) (3.6,1.4)};
            \draw [-triangle 60, dashed] plot[smooth, tension=.7] coordinates {(1,1.4) (2.2,1) (3.6,1)};
            \draw [-triangle 60, dashed] plot[smooth, tension=.7] coordinates {(-1.8,0.4) (0.6,-0.2) (2.2,0.4) (3.6,0.4)};
            \draw [-triangle 60, dashed] plot[smooth, tension=.7] coordinates {(-1.8,-0.2) (0.4,-0.8) (1.6,-0.4) (3.6,-0.4)};
            \draw [-triangle 60, dashed] plot[smooth, tension=.7] coordinates {(-1.8,-0.8) (0.4,-1.4) (2.2,-1) (3.6,-1)};
        
        \end{tikzpicture}
        
        \caption{The construction of a path-system that joins $O$ to $Y$ in $D-v$.}\label{f:joinOtoY}
    \end{figure}
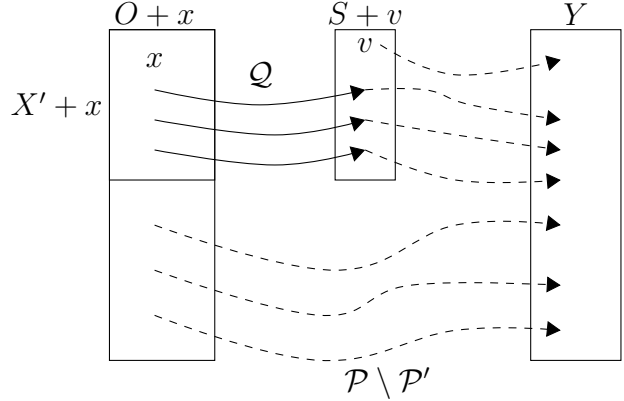
    
    However, since~$X'+x$ is incompressible to~${S + v}$, and~${S + v}$ separates~$Y$ from~$X'+x$, the subset of paths starting at vertices of~$X'+x$
    \[
        \mathcal{P}' = \{ P \in \mathcal{P} \colon V(P) \cap (X'+x) \neq \emptyset \}
    \]
    is orthogonal to~${S + v}$.
    
    In particular, no path in ${\mathcal{P} \setminus \mathcal {P}'}$ meets~${S + v}$, and therefore none of these paths meets any vertex from which~$Y$ is separated by~${S+v}$ in~$D$.
    
    Let ${\mathcal{Q} \in \mathfrak{J}_{D-v}(X', S)}$. 
    By the above comment, no path in~$\mathcal{Q}$ meets any path in~${\mathcal{P} \setminus \mathcal {P}'}$ 
    and so we can form a set of ${(X',Y)}$-paths disjoint from the paths in~${\mathcal{P} \setminus \mathcal {P}'}$ 
    by extending each path in~$\mathcal{Q}$ by the terminal segment of a path in~$\mathcal{P}'$ after~$S$.
    
    This set of paths together with~${\mathcal{P} \setminus \mathcal {P}'}$ witness 
    that~$O$ is joinable to~${Y - v}$ in~${D - v}$. 
\end{proof}

Using this we can show the following lemma.

\begin{lem}
    \label{lem:incomp-infinite-intersection}
    Assume that a countable~${X' \subseteq X}$ is ${(X,Y)}$-finitely extendable in~$D$. 
    Then for every family~${(V_i \colon i \in \mathbb{N})}$ of infinite subsets of~${X \setminus X'}$ 
    there is some~$O$ with~${X' \subseteq O \subseteq X}$ 
    such that~$O$ is joinable to~$Y$ in~$D$
    and~${O \cap V_i \neq \emptyset}$ for all~${i \in \mathbb{N}}$.
\end{lem}

\begin{proof}
    We build the desired ${(X,Y)}$-path-system by recursion. 
    The key to accomplish each step is the following claim.
    
    \begin{claim}
        \label{claim:intersect-all-inf}
        For every~${x \in X}$ there exists an ${(x, Y)}$-path~$P$ 
        for which there is a finite vertex set~${W \supseteq V(P)}$ with~${W \cap X' \subseteq \{ x \}}$ 
        such that the deletion of~$W$ preserves the conditions of Lemma~\ref{lem:incomp-infinite-intersection} for the remaining system, i.e., 
        ${(X' \setminus W)}$ is ${(X \setminus W, Y\setminus W)}$-finitely extendable in~${D - W}$. 
    \end{claim}
    
    \begin{proof}
        Since~${X' \subseteq X}$ is ${(X,Y)}$-finitely extendable in~$D$ we can pick a ${\mathcal{P} \in \mathfrak{J}(X'+x, Y)}$. 
        Let~$P$ be the unique element of~$\mathcal{P}$ with first vertex~$x$. 
        We obtain~$W$ by applying Lemma~\ref{lem:incomp-delete-one} with~${U = V(P)}$ and~$X'$ if~${x \notin X'}$. 
        If~${x \in X'}$, then we apply Lemma~\ref{lem:incomp-delete-one} with~${V(P)}$ and~${X' - x}$.
    \end{proof}
    
    First we assume that~${|X'| = \aleph_0}$, and we enumerate~${X' = \{ x_k \colon k \in \mathbb{N} \}}$. 
    We will build an ${(X,Y)}$-path-system ${\mathcal{P} = (P_i \colon i\in \mathbb{N})}$ 
    where~$P_{2k}$ starts at~$x_k$ and~$P_{2k+1}$ starts in~$V_k$. 
    We also will maintain a system~$D_n$,~$X_n$,~$Y_n$ obtained from
    $D$,~$X$,~$Y$ by the deletion of a finite vertex set containing~${\bigcup_{i<n}V(P_i)}$. 
    We will also demand that if~${n \leq 2k}$, then~${x_k \in X_n}$, and, moreover, that~${X' \cap X_n}$ is ${(X_n, Y_n)}$-finitely extendable in~$D_n$. 
    Then~${O := V^-(\mathcal{P})}$ will satisfy the Lemma~\ref{lem:incomp-infinite-intersection}.
    
    At the beginning we take ${D_0 := D}$,~${X_0 := X}$,~${Y_0 := Y}$ and we do not have any paths defined. 
    In step~$n$ we apply Claim~\ref{claim:intersect-all-inf} with~$D_n$,~$X_n$,~$Y_n$ and~${X' \cap X_n}$ choosing~$x$ to be~$x_{2k}$ if~${n = 2k}$ and an arbitrary element of~${V_k \cap X_n}$ if~${n = 2k+1}$. 
    This yields the desired path~$P_n$. 
    The triple $D_{n+1}$, $X_{n+1}$, $Y_{n+1}$ is obtained by the deletion of~$V(P_n)$ together with the finitely many extra vertices given by Claim~\ref{claim:intersect-all-inf}.
    
    For a finite~$X'$ we proceed similarly except only~${\left| X' \right|}$ many steps are devoted to join~$X'$ to~$Y$.
\end{proof}

\subsection{Applications of incompressibility}
\label{consequences}
\ 

For this subsection, 
let~$D$ be a digraph rooted in~$r$. 

\begin{cor}
    \label{cor:extend-intersect-family}
    Let~${I \in \mathcal{G}_D(w)}$ and suppose that~${J \in \mathcal{G}_D(w)}$ whenever~${I \subseteq J \subseteq \ingoing_D(w)}$ with~${J \setminus I}$ finite.
    Then for every countable family~$\mathcal{I}$ of infinite subsets of~${\ingoing_D(w) \setminus I}$, there is a~${I_{0}^{*} \in \mathcal{G}_D(w)}$ such that~${I_{0}^{*} \supseteq I}$ and~$I_0^{*}$ intersects every element of~$\mathcal{I}$.
\end{cor}

\begin{proof}
    We may assume without loss of generality that~${rw \notin E(D)}$. 
    Let~$A$ be the auxiliary digraph that we obtain from~$D$ by the following way. 
    We subdivide each~${e \in \ingoing_D(w)}$ with a new vertex~$x_e$ and each~${e \in \outgoing_D(r)}$ with a new vertex~$y_e$. 
    Then we delete~$r$ and~$w$ and reverse all the edges of the resulting digraph. 
    Let~${X := \{ x_e \colon e \in \ingoing_D(w) \}}$ and 
    let~${Y := \{ y_e \colon e \in \outgoing_D(r) \}}$. 
    Now every subset~$I'$ of~$\ingoing_D(w)$ corresponds to a subset~${X_{I'} := \{ x_e \colon e \in I' \}}$ of~$X$. 
    Then~$X_I$ is is finitely ${(X,Y)}$-extendable, 
    and~$\mathcal{I}$ corresponds to a countable family of subsets of~${X \setminus X_I}$. 
    The lemma follows by applying Lemma~\ref{lem:incomp-infinite-intersection}. 
\end{proof}

\begin{lem}
    \label{lem:bubble}
    Suppose that~${I \in \mathcal{G}_D(w)}$ such that~${(I+f) \in \mathcal{G}_D(w)}$ for every~${f \in \ingoing_D(w) \setminus I}$. 
    Assume that there is a~${uv \in E(D)}$ with~${u \neq r, v \neq w}$ for which~${I \notin \mathcal{G}_{D-uv}(w)}$. 
    Then there exists a set~${S \subseteq V(D) - r}$ containing~$v$ and an~${(r,S)}$-path-system~$\mathcal{P}$ with ${V^{+}(\mathcal{P}) = S}$; 
    such that~$S$ separates the tails of~${\ingoing_D(v) - u}$ from~$r$. 
    In particular,~$uv$ is the last edge of some~${P \in \mathcal{P}}$.
\end{lem}

\begin{proof}
    We may assume that~${rw \notin E(D)}$ since otherwise we apply the lemma  with the digraph~${D := D - rw}$ and edge set~${I := I - rw}$ instead and extend the resulting~$\mathcal{P}$ by the trivial path $rw$ unless it satisfies already the conditions. 
    
    Suppose~$\mathcal{P}_I$ witnesses~${I \in \mathcal{G}_D(w)}$. 
    The edge~$uv$ must be in one of the paths, say in~$P_{i_0} \in \mathcal{P}_I$ where $i_0\in I$ is the last edge of $P_{i_0}$. 
    Let $A$,~$X$,~$Y$ and~$X_I$ as in the proof of Corollary~\ref{cor:extend-intersect-family}. 
    Note that~${X_I - x_{i_0}}$ is joinable to~$Y$ in~${A - vu}$ and~$X_I$ is joinable to~$Y$ in~$A$ but not in~${A - vu}$. 
    It follows from Lemma~\ref{lem:incomp} that there is some~${S' \in \mathfrak{S}_{A-vu}(X_I,Y)}$ such that~${X_I - x_{i_0}}$ is incompressible to~$S'$ in~${A - vu}$. 
    We may assume that~${S' \cap (X_I \cup Y) = \emptyset}$, 
    since otherwise we replace in~$S'$ each~${x \in S' \cap X_I}$ by its unique out-neighbour in~$A$ and each~${y \in S' \cap Y}$ by its unique in-neighbour in~$A$. 
    Now~${X_I - x_{i_0}}$ is also incompressible to~$S'$ in~$A$ by Lemma~\ref{lem:incomp-finite-difference}, 
    but~$S'$ does not separate~$Y$ from~$X_I$ in~$A$ since otherwise~$X_I$ would not be joinable to~$Y$. 
    Hence there is a ${(u,Y)}$-path in~${A - vu}$ avoiding~$S'$, but no ${(v,Y)}$-path in~${A - vu}$ avoiding~$S'$. 
    Furthermore, ${S'' := S' + v}$ separates~$Y$ from~$X_I$ in~$A$. 
    By Lemma~\ref{lem:incomp-remains},~$X_I$ is incompressible to~$S''$. 
    Since~$X_I$ is joinable to~$Y$, it follows that~${S'' \in \mathfrak{S}_{A}(X_I,Y)}$. 
    
    By translating back the results to the original digraph we have the following conclusions.
    \begin{itemize}
        \item Every path-system witnessing~${I \in \mathcal{G}_D(w)}$ must be orthogonal to~$S''$. 
        \item For any in-neighbour~${u' \neq u}$ of~$v$, any ${(r,u')}$-path~$Q$ in~$D$ which avoids~$S''$ (should one exist) must necessarily contain~$w$. 
    \end{itemize}
    
    If~$S''$ separates the tails of~${\ingoing_D(v) - u}$ from~$r$  and~$\mathcal{Q}$ is a path-system witnessing~${I \in \mathcal{G}_D(w)}$, then the set~$\mathcal{P}$ of initial segments of the paths  in~$\mathcal{Q}$ until~$S''$ satisfies the conclusion of the lemma with~${S := S''}$.
    
    Otherwise, let~${f \in \ingoing_D(w) \setminus I}$ be arbitrary. 
    By the second condition~${S'' + w}$ separates the tails of~${\ingoing_D(v) - u}$ from~$r$ (see Figure~\ref{f:separatorandpathsystem}) and by assumption there is some path-system~$\mathcal{R}$ witnessing that~${(I + f) \in \mathcal{G}_D(w)}$. 
    Since the set of paths in~$\mathcal{R}$ with last edge in~$I$ is orthogonal to~$S''$, the set~$\mathcal{P}$ of initial segments of the paths in~$\mathcal{R}$ until~${S'' + w}$ is as desired where~${S := S'' + w}$.
    
    \begin{figure}[h]
        \centering
        \begin{tikzpicture}[scale=1.7]
            
            \draw [-triangle 60] (2,1) rectangle (3,-2.4);
            \node (v10) at (-0.4,-0.8) {};
            \node at (-0.6,-0.8) {$r$};
            \node (v4) at (2.6,0.6) {};
            \node (v4') at (2.4,0.6) {};
            \node (v14) at (2.5,0.6) {$w$};
            
            \node (v2) at (2.4,-2) {$v$};
            \node (v1) at (1.2,-2) {};
            \node (v1') at (1.3,-2) {};
            \node at (1.3,-2) {$u$};
            \draw [-triangle 60] (v1') edge (v2);
            \node at (2.4,1.2) {$S''+w$};
            \node at (3.6,0.55) {$I$};
            \node at (4,-0.9) {$\mathcal{R}$};
            \node at (1.6,0.8) {$f$};
            \node at (3.4,0.95) {$i_0$};
            \node at (0.2,-1.4) {$\mathcal{P}$};
            
            \node (v3) at (4,0.8) {};
            \node (v5) at (4,0.2) {};
            \node (v6) at (4,-0.2) {};
            \draw [-triangle 60] (v3) edge (v4);
            \draw [-triangle 60] (v5) edge (v4);
            \draw [-triangle 60] (v6) edge (v4);

            \node (v8) at (3.9,-1.6) {};
            \node (v9) at (3.9,-2.2) {};
            \node (v7) at (2.4,-1) {};
            \node (v13) at (2.8,-0.2) {};

            \draw [-triangle 60, dotted] (v7) edge (v2);
            \draw [-triangle 60, dotted] (v8) edge (v2);
            \draw [-triangle 60, dotted] (v9) edge (v2);
            
            \draw [-triangle 60] plot[smooth, tension=.7] coordinates {(v10) (0.8,0.4) (v4')};
            
            \draw [-triangle 60] plot[smooth, tension=.7] coordinates {(v10) (0.6,-0.4) (1.6,-0.6) (2.4,-0.2)};
            \draw [-triangle 60] plot[smooth, tension=.7] coordinates {(-0.4,-0.8) (0.8,-1.2) (1.6,-1) (2.4,-1)};
            
            \draw [-triangle 60] plot[smooth, tension=.7] coordinates {(-0.4,-0.8) (0,-2.2) (v1)};
            \node (v11) at (1.4,0.2) {};
            \node (v12) at (1.5,-0.4) {};
            \draw [-triangle 60, dotted] (v11) edge (v4');
            \draw [-triangle 60, dotted] (v12) edge (v4');
            
            \draw [-triangle 60, dashed] plot[smooth, tension=.7] coordinates {(v2) (3.4,-1.2) (4.4,-1) (5.2,0) (5,0.8) (v3)};
            \draw [-triangle 60, dashed] plot[smooth, tension=.7] coordinates {(2.4,-1) (4.2,-0.6) (4.6,0) (4.4,0.2) (v5)};
            \draw [-triangle 60, dashed] plot[smooth, tension=.7] coordinates {(2.4,-0.2) (3.6,-0.6) (4.2,-0.4) (4.4,-0.2) (v6)};
        \end{tikzpicture}
        
        \caption{The separator $S''+w$ and path-system $\mathcal{P}$.}
    \end{figure}
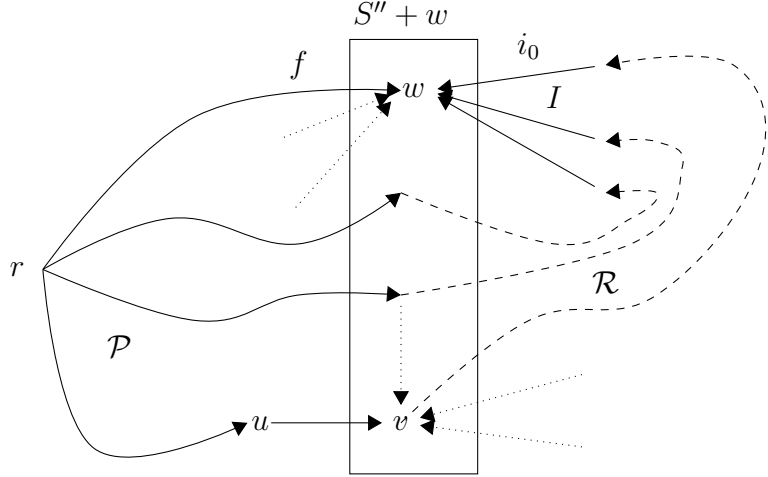\label{f:separatorandpathsystem}
\end{proof}

\section{Proof of Lemma~\ref{lem:key}}\label{sec:key-lemma}
We will use our results about incompressibility from Subsection~\ref{consequences} to prove Lemma~\ref{lem:key}.

\key*

\begin{proof}
    We may assume without the loss of generality that~${\outgoing_D(r) \subseteq E(G)}$ since otherwise we may add these edges to~$G$ while $D$ remains a $G$-quasi-flame for the new $G$. 
    Firstly, let us analyse how~${D \upharpoonright_v J}$ can fail to be a $G$-quasi-flame for some~${J \in \mathcal{G}_D(v)}$ with~${J \supseteq\ingoing_G(v)}$. 
    Since any path-system witnessing~${J \in \mathcal{G}_D(v)}$ is contained in~${D \upharpoonright_v J}$, the $G$-quasi-flame property cannot be ruined at~$v$. 
    We define~${D_0 := D \upharpoonright_v \ingoing_G(v)}$. 
    Let us say a set of edges~$I$ is \emph{relevant} if there is some vertex~${w \neq v}$ 
    such that~${\ingoing_G(w) \subseteq I \subseteq \ingoing_D(w)}$ with~${|I \setminus \ingoing_G(w)| < \aleph_0}$ and~${I \notin \mathcal{G}_{D_0}(w)}$. 
    Note that there are only countably many relevant sets. 
    Any path-system witnessing~${I \in \mathcal{G}_{D}(w)}$ for some relevant $I$ necessarily uses at most one edge from~${\ingoing_D(v) \setminus \ingoing_G(v)}$. 
    Therefore for every relevant~$I$ the set 
    \[
        N_I := \{ e \in \ingoing_{D}(v) \setminus \ingoing_{G}(v) \colon I \in \mathcal{G}_{D_0 + e}(w) \}
    \]
    is non-empty. 
    Moreover, for a~${J \in \mathcal{G}_{D}(v)}$ with~${J \supseteq \ingoing_{G}(v)}$ the following statements are equivalent:
    
    \begin{itemize}
        \item ${D \upharpoonright_v J}$ is an a $G$-quasi-flame;
        \item ${N_I \cap J \neq \emptyset}$ for every relevant~$I$.
    \end{itemize}
     
    By applying Corollary~\ref{cor:extend-intersect-family} with~$D$, ${\ingoing_G(v) \in \mathcal{G}_{D}(v)}$ and~${\{ N_I \colon |N_I| = \infty \}}$, 
    we obtain first an~${I_0^* \in \mathcal{G}_{D}(v)}$ with~${I_0^* \supseteq \ingoing_{G}(v)}$ 
    such that ${N_I \cap I_0^* \neq \emptyset}$ whenever~$N_I$ is infinite. 
    Let~$I^{*}$ be a superset of~$I_0^{*}$ that intersects all the finite~${N_I}$ and is $\subseteq$-minimal with respect to this property. 
    Clearly, the definition of~$I^*$ ensures the $G$-quasi-flame property of~${D^* := D \upharpoonright_v I^*}$ at all the vertices other than~$v$. 
    It remains to show that~${I^* \in \mathcal{G}_{D^*}(v)}$. 
    
    We  assume $I^* \supsetneq I_0^*$, since otherwise there is nothing to prove, and fix an enumeration~${I^* \setminus I_0^* = \{ e_k \colon 0 \leq k < |I^* \setminus I_0^*| \}}$. 
    For every~${u_i v_i = e_i \in I^* \setminus I_0^{*}}$, 
    by the minimality of~$I^*$ there is some relevant~${I_{i} \subseteq \ingoing_L(w)}$ for some~${w \neq v_i}$  
    such that~${I_{i} \in \mathcal{G}_{D^*}(w)}$ but~${I_{i} \not\in \mathcal{G}_{D^* - e_i}(w)}$. 
    Recall that~$D^*$ has the $G$-quasi-flame property at~$w$. 
    Moreover, ${u_i \neq r}$ since otherwise~${rv_i \in \ingoing_G(v_i) \subseteq I_0^*}$ would contradict~${e_i \in I^* \setminus I_0^{*}}$. 
    Therefore we can apply Lemma~\ref{lem:bubble} with~$D^*$, ${I_{i} \in \mathcal{G}_{D^*}(w)}$ and the edge~$e_i$, and we denote by~$S'_{i}$ and~$\mathcal{P}_{i}$ the resulting set and ${(r,S'_{i})}$-path-system, respectively. 
    
    By using the path-systems~$\mathcal{P}_i$ for~${e_i \in I^* \setminus I_0^*}$ we will build sets~$\mathcal{Q}$ and~$\mathcal{R}$ of paths in~$D^*$ with the following properties. 
    \begin{enumerate}
        [label=(\roman*)]
        \item\label{item:separate} ${S := V^{+}(\mathcal{Q} \cup \mathcal{R})}$ separates the tails of the edges in~$I_0^*$ from~$r$ in~$D^*$ where~${r \notin S}$;
        \item $\mathcal{Q}$ is an ${(r,v)}$-path-system with~${E^+(\mathcal{Q}) = I^{*} \setminus I_0^{*}}$; 
        \item $\mathcal{R}$ is an ${(r,S-v)}$-path-system; 
        \item ${V(R) \cap V(Q) = \{ r \}}$ for all~${Q \in \mathcal{Q}}$ and all~${R \in \mathcal{R}}$. 
    \end{enumerate}
    
    First we show that ${I^* \in \mathcal{G}_{D^*}(v)}$ if we have such sets~$\mathcal{Q}$ and~$\mathcal{R}$. 
    Let~$\mathcal{P}$ be a path-system witnessing~${I_0^* \in \mathcal{G}_{D^*}(v)}$. 
    Then each~${P \in \mathcal{P}}$ has a last common vertex~$v_P$ with~$S-v$ by~\ref{item:separate}. 
    We extend the unique~${R \in \mathcal{R}}$ that terminates at~$v_P$ with the terminal segment of~$P$ from~$v_P$ to~$v$. 
    The set of these extended paths united with~$\mathcal{Q}$ witness~${I^* \in \mathcal{G}_{D^*}(v)}$. 
    
    So let us build the path-systems~$\mathcal{Q}$ and~$\mathcal{R}$. 
    We construct two sequences~${(\mathcal{Q}_k \colon 0 \leq k < |I^* \setminus I_0^*|)}$ and~${(\mathcal{R}_k \colon 0 \leq k < |I^* \setminus I_0^*|)}$ of sets of paths such that:
   
    \begin{enumerate}
        [label=(\arabic*)]
        \item\label{condition4} ${S_k := V^{+}(\mathcal{Q}_k \cup \mathcal{R}_k)}$ separates the tails of the edges~${\ingoing_{D^*}(v) \setminus \{ e_i \colon 0 \leq i \leq k \}}$ from~$r$ in~$D^*$ where~${r \notin S_k}$;
        \item\label{condition1} $\mathcal{Q}_k$ is an $(r,v)$-path-system with~${E^+(\mathcal{Q}_k) = \{ e_i \colon 0 \leq i \leq k \}}$; 
        \item\label{condition2} $\mathcal{R}_k$ is an ${(r,S_k-v)}$-path-system;  
        \item\label{condition3} ${V(R) \cap V(Q) = \{ r \}}$ for all~${Q \in \mathcal{Q}_k}$ and all~${R \in \mathcal{R}_k}$; 
        \item\label{condition5} ${\mathcal{Q}_{k} \subseteq \mathcal{Q}_{k+1}}$; and
        \item\label{condition6} each path in~$\mathcal{R}_{k+1}$ extends some path in~$\mathcal{R}_k$ (not necessarily properly). 
    \end{enumerate}
    
    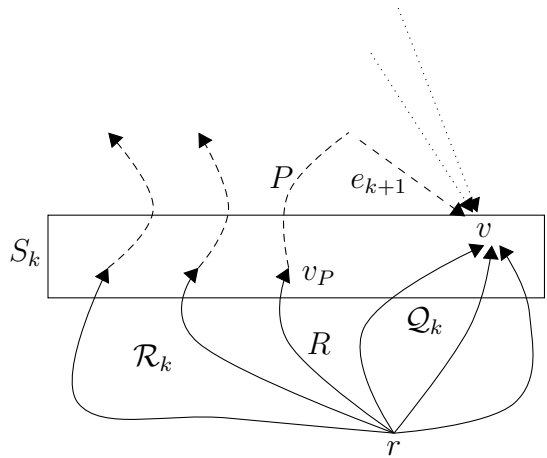
\begin{figure}[h]
        \centering
        \begin{tikzpicture}
            \node (v2) at (3.8,-0.7) {$v$};
            \node (v3) at (2.2,1.8) {};
            \node (v1) at (2.6,2.4) {};
            \node (v4) at (2,0.6) {};
            \node at (1.6,-1.3) {$v_P$};
            \node at (1.6,-2.2) {$R$};
            \node at (1.1,0) {$P$};
            \node at (3,-1.9) {$\mathcal{Q}_k$};
            \node at (-0.6,-2.4) {$\mathcal{R}_k$};
            \draw [dotted,-triangle 60] (v1) edge (v2);
            \draw [dotted, -triangle 60] (v3) edge (v2);
            \draw [dashed, -triangle 60] (v4) edge (v2);

            \node (v5) at (2.6,-3.4) {};
            \node at (2.6,-3.6) {$r$};
            \draw [-triangle 60] plot[smooth, tension=.7] coordinates {(v5) (4.2,-3) (4.4,-1.8) (4.0,-0.9)};

            \draw [-triangle 60] plot[smooth, tension=.7] coordinates {(v5) (3.6,-2) (3.9,-0.9)};

            \draw [-triangle 60] plot[smooth, tension=.7] coordinates {(v5) (2.2,-2) (3.8,-0.9)};

            \draw [-triangle 60] plot[smooth, tension=.7] coordinates {(v5) (0.4,-3.2) (-1.6,-3) (-1.2,-1.2)};

            \draw [-triangle 60] plot[smooth, tension=.7] coordinates {(v5) (0,-2.2) (0,-1.2)};

            \draw [-triangle 60] plot[smooth, tension=.7] coordinates {(v5) (1.2,-2.2) (1.2,-1.2)};

            \draw [densely dashed] plot[smooth, tension=.7] coordinates {(1.2,-1.2) (1.2,-0.2) (v4)};
            \draw [densely dashed, -triangle 60] plot[smooth, tension=.7] coordinates {(-1.2,-1.2) (-0.6,-0.4) (-1.2,0.6)};
            \draw [densely dashed, -triangle 60] plot[smooth, tension=.7] coordinates {(0,-1.2) (0.4,-0.4) (0,0.6)};
            \draw (4.6,-0.5) rectangle (-2,-1.6);
            \node at (-2.3,-1) {$S_k$};
            \node at (2.4,-0.1) {$e_{k+1}$};
        \end{tikzpicture}
        
        \caption{The construction of~$\mathcal{Q}_{k+1}$ and~$\mathcal{R}_{k+1}$. The vertex~$v$ might have some of its in-neighbours in~$S_k$.}
    \end{figure} 

    For ${k = 0}$, let~$Q_0$ be the unique path in~$\mathcal{P}_{0}$ with~${E^+(Q_0) = \{e_0\}}$ (see Lemma~\ref{lem:bubble}). 
    Then~${\mathcal{Q}_0 := \{ Q_0 \}}$ and~${\mathcal{R}_0 := \mathcal{P}_{0} - Q_0}$ satisfy the conditions. 
    
    Let us assume that we have constructed $\mathcal{Q}_k$ and $\mathcal{R}_k$. 
    Recall that the tail of~$e_{k+1}$ is separated from~$r$ by~${S_k}$ because of~\ref{condition4}. 
    Thus, each~${P \in \mathcal{P}_{k+1}}$ whose terminal vertex is separated from~$r$ by~${S_k}$ meets~$S_k-v$. 
    Let us denote the last common vertex of such a~$P$ with~$S_k-v$ by~$v_{P}$. 
    We take the unique~${R \in \mathcal{R}_{k}}$ terminating~$v_{P}$ and extend it forward by the terminal segment of~$P$ from~$v_{P}$. 
    The unique newly constructed path with last edge~$e_{k+1}$   together with~$\mathcal{Q}_k$ define the set~$\mathcal{Q}_{k+1}$, 
    and the other (not necessarily properly) forward-extended paths in~$\mathcal{R}_{k}$  
    define~$\mathcal{R}_{k+1}$. 
    
    For all the conditions but~\ref{condition4} it follows from the construction directly that~$\mathcal{Q}_{k+1}$ and~$\mathcal{R}_{k+1}$ satisfy them. 
    To show the preservation of~\ref{condition4} we need the following lemma. 
    \begin{lem}\label{lem:separsup}
        Let~$D$ be a rooted digraph, ${x \in V(D) - r}$ and let~$\mathcal{S}$ be a non-empty collection of subsets of $V(D)-r$ where each 
        ~$S\in \mathcal{S}$ separates $x$ from $r$. 
        Then the set~${\bigvee \mathcal{S}}$ of those elements~$s$ of~${\bigcup \mathcal{S}}$ that are separated from~$r$ by every~${S \in \mathcal{S}}$ also separates $x$ from $r$.
    \end{lem}
    
    \begin{proof}
        Suppose for a contradiction that some ${(r,x)}$-path~$P$ avoids~${\bigvee \mathcal{S}}$. 
        Let~${S_0 \in \mathcal{S}}$ be arbitrary and let~$s_0$ be the last common vertex of~$P$ with~$S_0$. 
        We conclude that~${s_0 \notin \bigvee \mathcal{S}}$. 
        Hence, there is some~${S_1 \in \mathcal{S}}$ that does not separate~$s_0$ from~$r$. 
        Then the last common vertex~$s_1$ of~$P$ with~$S_1$ is strictly later on~$P$ than~$s_0$. 
        By continuing the construction recursively we end up with infinitely many pairwise distinct~${s_n \in V(P)}$ which is a contradiction.
    \end{proof}
    
    Both~$S'_{k+1}$ and~${S_k}$ separate the tails of the edges in the set~${\ingoing_{D^*}(v) \setminus \{ e_i \colon 0 \leq i \leq k+1 \}}$ from~$r$ in~$D^*$. 
    Moreover, ${s \in S'_{k+1} \setminus S_{k+1}}$ only if~$s$ is not separated from~$r$ by~$S_k$, and hence 
    \[
        \bigvee \{ S'_{k+1}, S_k \} \subseteq S_{k+1}. 
    \]
    Therefore~\ref{condition4} follows by applying Lemma~\ref{lem:separsup} to the tails of the edges~$e_i$ for all~${i > k+1}$ separately. 
    
    If ${\left|I^* \setminus I_0^*\right| = k+1 \in \mathbb{N}}$, then~$\mathcal{Q}_k$ and~$\mathcal{R}_k$ are the desired~$\mathcal{Q}$ and~$\mathcal{R}$. 
    Suppose that ${\left|I^* \setminus I_0^*\right|=\aleph_0}$. 
    Then let~${\mathcal{Q} := \bigcup_{k \in \mathbb{N}}\mathcal{Q}_{k}}$ and~${\mathcal{R} := \bigcup_{m \in \mathbb{N}} \bigcap_{m \leq k \in \mathbb{N}} \mathcal{R}_{k}}$. 
    For all the conditions but~\ref{item:separate} it follows from the construction directly that~$\mathcal{Q}$ and~$\mathcal{R}$ satisfy them. 
    By Lemma~\ref{lem:separsup}, 
    ${\bigvee \{ S_k \colon k \in \mathbb{N} \}}$ separates the tails of the edges in the set~$I_0^* $ from~$r$ in~$D^*$. 
    Note that if an~${R \in \mathcal{R}_k}$ is properly extended to a path in~$\mathcal{R}_{k+1}$ in step~$k$ then its last vertex is not separated by~$S_{k+1}$ from~$r$.
    Thus~${\bigvee \{ S_k \colon k \in \mathbb{N} \} \subseteq V^{+}(\mathcal{R})}$ from which~\ref{item:separate} follows. 
\end{proof}

\end{document}